\documentclass[12pt,a4paper,reqno]{amsart}
\usepackage[utf8]{inputenc}
\usepackage[T1]{fontenc}       
\usepackage{ae}
\usepackage[english]{babel}   
\usepackage{amsfonts}         
\usepackage{amsmath}
\usepackage{amsthm}
\usepackage{amssymb}
\usepackage{mathtools}
\usepackage{calc}
\usepackage{centernot}
\usepackage{color}
\usepackage[pdftex]{hyperref}
\usepackage{mathrsfs}

\usepackage{bigints} 

\usepackage{etoolbox}
\makeatletter
\patchcmd{\@maketitle}
  {\ifx\@empty\@dedicatory}
  {\ifx\@empty\@date \else {\vskip3ex \centering\footnotesize\@date\par\vskip1ex}\fi
   \ifx\@empty\@dedicatory}
  {}{}
\patchcmd{\@adminfootnotes}
  {\ifx\@empty\@date\else \@footnotetext{\@setdate}\fi}
  {}{}{}
\makeatother

\usepackage{mathpazo} 

\usepackage{xargs}

\newcommandx{\huom}[2][1=]{\todo[linecolor=red,backgroundcolor=red!10,bordercolor=red,#1]{#2}}

\newcommand{\ve}{\varepsilon}
\newcommand{\N}{\mathbb{N}}
\newcommand{\C}{\mathbb{C}}
\newcommand{\R}{\mathbb{R}}

\newcommand{\HH}{\mathbb{H}}
\newcommand{\Ss}{\mathbb{S}}
\newcommand{\abs}[1]{\lvert #1 \rvert}

\newcommand{\ang}[1]{\left\langle #1 \right\rangle}

\newcommand{\loc}{{\rm loc}}

\newcommand{\p}{\partial}

\newcommand{\Int}{\text{Int }}

\DeclareMathOperator\dist{dist}
\DeclareMathOperator\dv{div}

\newcommand{\an}{\sphericalangle}
\newcommand{\pinf}{\partial_{\infty}}

\DeclareMathOperator\vol{vol}

\numberwithin{equation}{section}

\theoremstyle{plain}
\newtheorem{thm}{Theorem}[section]

\newtheorem{prob}[thm]{Problem}
\newtheorem*{thm*}{Theorem}

\theoremstyle{definition}
\newtheorem{defin}[thm]{Definition}

\author{Esko Heinonen}
\address{E. Heinonen,
Departamento de Geometr\'ia y Topolog\'ia, 
Universidad de Granada, 18071 Granada, Spain.
}
\email{ea.heinonen@gmail.com}

\thanks{E.H. supported by a grant of the Finnish Academy of Science and Letters.}


\title[Asymptotic Dirichlet problem]{Survey on the asymptotic Dirichlet problem for the minimal surface equation}

\begin{document}

\maketitle

\begin{abstract}
We give a survey on the development of the study of the asymptotic Dirichlet problem for the minimal surface equation on Cartan-Hadamard manifolds. Part of this survey is based on the introductory part of the doctoral dissertation \cite{He-thesis} of the author. The paper is organised as follows. First we introduce Cartan-Hadamard manifolds and the concept of asymptotic Dirichlet problem, then discuss about the development of the results and describe the methods used in the proofs. In the end we mention some results about the nonsolvability of the asymptotic Dirichlet problem.
\end{abstract}
%


\section{Preliminaries}

\subsection{Cartan-Hadamard manifolds}

A \emph{Cartan-Hadamard} (also \emph{Ha\-damard}) manifold $M$ is a complete simply connected Riemannian manifold whose all sectional curvatures satisfy
 	\[
	K_M \le 0.
	\]
 The most simple examples of such manifolds are the Euclidean space $\R^n$, with zero curvature, and the hyperbolic space $\HH^n$, with constant negative curvature $-1$. The name of these manifolds has its origin in the Cartan-Hadamard theorem which states that the exponential map is a diffeomorphism in the whole tangent space at every point of $M$. 
 
 Another important examples of Cartan-Hadamard manifolds are given by the rotationally symmetric model manifolds with radial curvature function $f$ satisfying $f''\ge0$. We recall that a model manifold is $\R^{n}$ equipped with the rotationally symmetric metric that can be written as
 	\[
	g = dr^2 + f(r)^2d\vartheta^2,
	\]
 where $r$ (as in the rest of the paper) is the distance to a pole $o$ and $d\vartheta$ is the standard metric on the unit sphere $\Ss^{n-1}$. To justify the requirement $f''\ge0$, we note that the sectional curvatures of a model manifold can be obtained from the radial curvature function, namely we have 
 	\begin{equation}\label{rotsym_sect_for}
	K_{M_f}(P_x) = -\frac{f''\big(r(x)\big)}{f\big(r(x)\big)} \cos^2 \alpha + \frac{1-f'\big(r(x)\big)^2}
	{f\big(r(x)\big)^2}\sin^2\alpha,
	\end{equation}
where $\alpha$ is the angle between $\nabla r(x)$ and the 2-plane $P_x\subset T_xM$. In the case of the radial sectional curvature the formula simplifies to 
	\[
	K_{M_f} = -\frac{f''}{f}.
	\]
For the verification of these formulae one could see e.g. \cite{Va_lic}, or \cite{BO} where more general formula was given for the sectional curvatures of warped product manifolds.

The radial curvature function $f$ is an example of solutions to the Jacobi equation that is defined as follows. Given a smooth function $k\colon [0,\infty)\to[0,\infty)$ we denote by $f_k\colon [0,\infty) \to \R$ the smooth non-negative solution to the initial value problem (Jacobi equation)
	\[
	\begin{cases}
	f_k(0)=0, \\
	f_k'(0) = 1,\\
	f_k'' = k^2 f_k.
	\end{cases}
	\]
These functions play an important role in estimates involving curvature bounds since the resulting model manifolds can be used in various comparison theorems, e.g. Hessian and Laplace comparison (see \cite{GW}).


 \subsection{Asymptotic Dirichlet problem on Cartan-Hadamard manifolds}\label{intro_ADP}

 Cartan-Hadamard manifolds can be compactified by adding the \emph{asymptotic boundary} (also 
 \emph{sphere at infinity}) $\pinf M$ and equipping the resulting space $\bar M \coloneqq M \cup \pinf M$
 with the \emph{cone topology}. This makes $\bar M$ homeomorphic to the closed unit ball. 
 The asymptotic boundary $\pinf M$ consists of equivalence classes of geodesic 
 rays under the equivalence relation 
 	\[
	\gamma_1 \sim \gamma_2 \quad \text{if} \quad \sup_{t\ge0} \dist\big(\gamma_1(t),\gamma_2(t)\big) <\infty.
	\] 
Equivalently it can be considered as the set of geodesic rays emitting from a fixed point $o\in M$, when each ray corresponds to a point on the unit sphere of $T_oM$, and this justifies the name sphere at infinity.
 
The basis for the cone topology in $\bar M$ is formed by cones
	\[
	C(v,\alpha) \coloneqq \{y\in M\setminus \{x\} \colon \an(v,\dot\gamma_0^{x,y}) < \alpha \},
	\quad v\in T_xM, \, \alpha>0,
	\]
truncated cones
	\[
	T(v,\alpha,R) \coloneqq C(v,\alpha) \setminus \bar B(x,R), \quad R>0,
	\]
and all open balls in $M$. Here $\gamma^{x,y}$ denotes the unique geodesic joining $x$ to $y$, $\dot\gamma_0$ is the initial unit vector of geodesic $\gamma$, and $\an(\cdot,\cdot)$ the angle between two vectors. Cone topology was first introduced by P. Eberlein and B. O'Neill in \cite{EO}.

 Now we can define the \emph{asymptotic Dirichlet problem} (also \emph{Dirichlet problem at infinity}). Even though we will consider the minimal surface equation 
 	\begin{equation}\label{mingraph-eq}
	\dv \frac{\nabla u}{\sqrt{1+|\nabla u|^2}} = 0,
	\end{equation}
 we formulate the problem for a general quasilinear elliptic operator $Q$. Concerning terminology, the minimal surfaces given as graphs of solutions to \eqref{mingraph-eq} will be called also \emph{minimal graphs}.
 	\begin{prob}\label{AD-prob}
	Let $\theta\colon \pinf M \to \R$ be a continuous function. Does there exist a continuous function
	$u\colon \bar M \to \R$ with    
	\[
	\begin{cases}
	Q[u] = 0 \quad \text{in } M; \\
	u|\pinf M = \theta,
	\end{cases}
	\]
	and if yes, is the function $u$ unique?
	\end{prob}
 In the case such function $u$ exists for every $\theta \in C(\pinf M)$, we say that the asymptotic Dirichlet problem in $\bar M$ is \emph{solvable}. We will see that the solvability of this problem depends heavily on the geometry of the manifold $M$, but the uniqueness of the solutions depends also on the operator $Q$. Namely, for the usual Laplace, $\mathcal{A}$-harmonic and minimal surface operators we have the uniqueness but in the case of more complicated operators, that do not satisfy maximum principles, the uniqueness of solutions will be lost (see e.g. \cite{CHH2}).


\section{Overview of the results}

The study of the asymptotic Dirichlet problem has its origin in the question of the existence of entire bounded nonconstant harmonic functions on Cartan-Hadamard manifolds. Namely, one way to prove the existence is to solve the asymptotic Dirichlet problem with arbitrary continuous boundary data on $\pinf M$. This question gained lot of interest after the seminal work by Greene and Wu \cite{GW}, where they conjectured that if the sectional curvatures of a Cartan-Hadamard manifold $M$ satisfy
	\[
	K_M \le \frac{C}{r^2}, \quad C>0,
	\]
outside a compact set, then there exists an entire bounded nonconstant harmonic function on $M$. 

From now on we will focus on the mean curvature equation; a brief discussion about the results concerning harmonic and $\mathcal{A}$-harmo\-nic functions can be found e.g. in \cite{He-thesis}. We will mention also some results that are not directly about the asymptotic Dirichlet problem but that have motivated the study or have had an important role in the proofs of other results.

\subsection{Minimal surfaces}

To begin with bounded domains $\Omega\subset\R^n$, we mention that the usual Dirichlet problem for the minimal graphs was solved in 1968 by Jenkins and Serrin \cite{JenkinsSerrin} under the assumption that $\p\Omega$ has nonnegative mean curvature. Serrin \cite{serrin} generalised this for the graphs with prescribed mean curvature, and more recently Guio and Sa Earp \cite{GuioSa} considered similar Dirichlet problem in the hyperbolic space.

First result about the existence of entire minimal graphs is due to Nelli and Rosenberg. In \cite{NR}, in addition to constructing catenoids, helicoids and Scherk-type surfaces in $\HH^2\times\R$, they proved the following theorem using the disk model of $\HH^2$.
	\begin{thm}\label{nelli-rosen-thm}
	Let $\Gamma$ be a continuous rectifiable Jordan curve in $\pinf \HH^2\times\R$, that is a vertical graph. 
	Then, there exists a minimal vertical graph on $\HH^2$ having $\Gamma$ as asymptotic boundary. 
	The graph is unique.
	\end{thm}

In 2005 Meeks and Rosenberg \cite{MR} developed the theory of properly embedded minimal surfaces in $N\times\R$, where $N$ is a closed orientable Riemannian surface but the existence of entire minimal surfaces in product spaces $M\times\R$ really draw attention after the papers by Collin and Rosenberg \cite{CR} and Gálvez and Rosenberg \cite{GR}. In \cite{CR} Collin and Rosenberg constructed a harmonic diffeomorphism from $\C$ onto $\HH^2$ and hence disproved the conjecture of Schoen and Yau \cite{ScYau}. Gálvez and Rosenberg generalised this result to Hadamard surfaces whose curvature is bounded from above by a negative constant. A key tool in their constructions was to solve the Dirichlet problem on unbounded ideal polygons with alternating boundary values $\pm\infty$ on the sides of the ideal polygons (so-called Jenkins-Serrin problem). In the end of \cite{GR} Gálvez and Rosenberg obtain also a counterpart of the Theorem \ref{nelli-rosen-thm} under the assumption $K_M \le C < 0$.

Sa Earp and Toubiana \cite{ET} constructed minimal vertical graphs over unbounded domains in $\HH^2\times\R$ taking certain prescribed finite boundary data and certain prescribed asymptotic boundary data. Espírito-Santo and Ripoll \cite{ER} considered the existence of solutions to the exterior Dirichlet problem on simply connected manifolds with negative sectional curvature. Here the idea is to find minimal hypersurfaces on unbounded domains with compact boundary assuming zero boundary values.

Espírito-Santo, Fornari and Ripoll \cite{EFR} proved the solvability of the asymptotic Dirichlet problem on Riemannian manifold $M$ whose sectional curvatures satisfy $K_M\le -k^2$, $k>0$, and under the assumption that there exists a point $p\in M$ such that the isotropy group at $p$ of the isometry group of $M$ acts transitively on the geodesic spheres centred at $p$.

Rosenberg, Schulze and Spruck \cite{RSS} studied minimal hypersurfaces in $N\times\R_+$ with $N$ complete Riemannian manifold having non-negative Ricci curvature and sectional curvatures bounded from below. They proved so-called half-space properties both for properly immersed minimal surfaces and for graphical minimal surfaces. In the latter, a key tool was a global gradient estimate for solutions of the minimal surface equation.

Ripoll and Telichevesky \cite{RT} showed the existence of entire bounded nonconstant solutions for slightly larger class of operators, including minimal surface operator, by studying the strict convexity (SC) condition of the manifold. Similar class of operators was studied also by Casteras, Holopainen and Ripoll \cite{CHR} but instead of considering the SC condition, they solved the asymptotic  Dirichlet problem by using similar barrier functions as in \cite{HoVa}. Both of these gave the existence of minimal graphic functions under the sectional curvature assumption 
	\begin{equation}\label{HoVaCurv2}
	-r(x)^{-2-\ve} e^{2kr(x)} \le K_M(P_x) \le -k^2
	\end{equation}
outside a compact set, and the latter also included the assumption
	\begin{equation}\label{HoVaCurv1}
	-r(x)^{2(\phi-2)-\ve} \le K(P_x) \le -\frac{\phi(\phi-1)}{r(x)^2},
	\end{equation}
$r(x)\ge R_0$, for some constants $\phi>1$ and $\ve,R_0>0$.

In \cite{CHR1} Casteras, Holopainen and Ripoll adapted a method that was earlier used by Cheng \cite{cheng} for the study of harmonic functions and by Vähäkangas \cite{Va2} for $\mathcal{A}$-harmonic functions, and proved  the following.
	\begin{thm}\label{CHR-op-thm}
	Let $M$ be a Cartan-Hadamard manifold of dimension $n\ge 3$ and suppose that
	\begin{equation}\label{CHRoptub}
	-\frac{\big(\log r(x)\big)^{2\tilde\ve}}{r(x)^2} \le K_M(P_x) \le -\frac{1+\ve}{r(x)^2\log r(x)}
	\end{equation}
	holds for all $2$-planes $P_x\subset T_xM$ and for some constants $\ve>\tilde\ve>0$ and $r$ large enough. Then the asymptotic Dirichlet
	problem for the minimal surface equation \eqref{mingraph-eq} is uniquely solvable.
	\end{thm}
The proof was based on the usage of Sobolev and Caccioppoli-type inequalities together with complementary Young functions.

Telichevesky \cite{telichevesky} considered the Dirichlet problem on unbounded domains $\Omega$ proving the existence of solutions provided that $K_M\le-1$, the ordinary boundary of $\Omega$ is mean convex and that $\Omega$ satisfies the SC condition at infinity. The SC condition was studied by Casteras, Holopainen and Ripoll also in \cite{CHR2} and they proved that the manifold $M$ satisfies the SC condition under very general curvature assumptions. As special cases they obtain the bounds \eqref{CHRoptub} and
	\begin{equation}\label{CHR-lowb}
	-ce^{(2-\ve)r(x)}e^{e^{r(x)/e^3}} \le K_M \le -\phi e^{2r(x)}
	\end{equation}
for some constants $\phi>1/4$, $\ve>0$ and $c>0$. In addition to the asymptotic Dirichlet problem, Casteras, Holopainen and Ripoll applied the SC condition to prove also the solvability of the asymptotic Plateau problem.

Adapting Vähäkangas' method Casteras, Heinonen and Holopai\-nen showed that, as in the case of $\mathcal{A}$-harmonic functions, the curvature lower bound can be replaced by a pointwise pinching condition obtaining the following result.
\begin{thm}\label{pinch_maintheorem}
  Let $M$ be a Cartan-Hadamard manifold of dimension $n\ge 2$ and let $\phi>1$. Assume that 
      \begin{equation}\label{p_curvassump}
       K(P_x) \le - \frac{\phi(\phi-1)}{r(x)^2},
      \end{equation}
holds for all $2$-planes $P_x\subset T_xM$
containing the radial vector $\nabla r(x)$, with $x\in M\setminus B(o,R_0)$. Suppose also
that there exists a constant $C_K<\infty$ such that 
      \begin{equation}\label{pinchassump}
       \abs{K(P_x)} \le C_K\abs{K(P_x')}
      \end{equation}
whenever $x\in M\setminus B(o,R_0)$ and $P_x,P_x'\subset T_xM$ are $2$-planes 
containing the radial vector $\nabla r(x)$. Moreover, suppose that the dimension $n$ and
the constant $\phi$ satisfy the relation
      \begin{equation*}\label{pinch_dimrestriction}
       n> \frac{4}{\phi} +1.
      \end{equation*}
Then the asymptotic Dirichlet problem for the minimal surface equation \eqref{mingraph-eq} is
uniquely solvable.
 \end{thm} 
It is worth to point out that choosing $\phi>4$ in Theorem \ref{pinch_maintheorem}, the result holds for every dimension $n\ge2$ and if $n\ge5$, $\phi$ can be as close to $1$ as we wish. Moreover, if $M$ is $2$-dimensional, the condition \eqref{pinchassump} is trivially satisfied and the asymptotic Dirichlet problem can be solved by assuming only the curvature upper bound \eqref{p_curvassump} and $\phi>4$.

In the case of rotationally symmetric manifolds it is possible to obtain solvability results without any assumptions on the lower bound of the sectional curvatures. Ripoll and Telichevesky \cite{RTgeomdedi} solved the asymptotic Dirichlet problem on $2$-dimensional surfaces, and later Casteras, Heinonen and Holopainen \cite{CHH3} gave a proof that holds in any dimension $n\ge2$, obtaining the following.
	\begin{thm}\label{thm1.5}
	Let $M$ be a rotationally symmetric $n$-dimensional Cartan-Hadamard manifold whose radial sectional curvatures outside a compact set 
	satisfy the upper bounds
	 \begin{equation*}
	K(P_x) \le - \frac{1 + \ve }{r(x)^2 \log r(x)}, \quad \text{if } n=2 
	 \end{equation*}
      	and
   	 \begin{equation*}
  	K(P_x) \le - \frac{1/2 + \ve }{r(x)^2 \log r(x)}, \quad \text{if } n\ge3. 	 
   	 \end{equation*}
     		Then the asymptotic Dirichlet problem for the minimal surface equation \eqref{mingraph-eq} is
	uniquely solvable.  
  \end{thm}
We point out that the curvature assumptions in Theorem \ref{thm1.5} are most likely optimal (even the constants in the numerators). Namely, March \cite{march} gave an if and only if result for the existence of bounded entire nonconstant harmonic functions under the same curvature assumptions (see also the discussion in Section \ref{intro_nonex}).

\subsection{Other prescribed mean curvature surfaces}\label{sec-other-op}

As we have already mentioned, the asymptotic Dirichlet problem has been previously considered also for other type of operators (starting from the Laplacian) but the class of surfaces, that is very closely related to minimal surfaces, are the surfaces of prescribed mean curvature. It is well known that the graph of a function $u\colon M\to \R$ has prescribed mean curvature $H$ if $u$ satisfies
	\[
	\dv \frac{\nabla u}{\sqrt{1+|\nabla u|^2}} = nH,
	\]
where $H\colon M\to\R$ is a given function. As it is reasonable to believe, in this case the solvability of the asymptotic Dirichlet problem will depend also on the function $H$.

A very special type of prescribed mean curvature surfaces are the so-called $f$-minimal surfaces that are obtained by replacing the function $nH$ by
	\[
	\ang{\bar\nabla f,\nu_u},
	\]
where $f\colon M\times\R \to \R$ and $\nu_u$ is the downward pointing unit normal to the graph of $u$. In this case the mean curvature will depend, not only on the point of the manifold, but also on the solution itself. The name $f$-minimal comes from the fact that these surfaces are minimal in the weighted manifolds $(M,g,e^{-f} d\vol_M)$, where $(M,g)$ is a complete Riemannian manifold with volume element $d\vol_M$.

In \cite{CHH2} Casteras, Heinonen and Holopainen studied these $f$-mini\-mal graphs and solved first the usual Dirichlet problem under suitable assumptions and then applied it to solve the asymptotic Dirichlet problem under curvature assumptions similar to those that appeared in \cite{CHR}, i.e. \eqref{HoVaCurv2} and \eqref{HoVaCurv1}. We point out that in \cite{CHH2} it was necessary to assume that the function $f\in C^2(M\times\R)$ is of the form
	\[
	f(x,t) = m(x) + r(t).
	\]
Other key assumptions are related to the decay of $|\bar\nabla f|$ compared to the curvature upper bound of the manifold. For example, it is required that
	\[
	\sup_{\p B(o,r)\times\R} |\bar\nabla f| = o \left( \frac{f'_a(r)}{f_a(r)} r^{-\ve-1} \right),
	\]
where $f_a$ is the Jacobi solution related to the curvature upper bound. Note that $(n-1)f_a'(r)/f_a(r)$ is the mean curvature of a sphere of radius $r$ on a model manifold with curvature $-a^2$. The assumptions related to the decay of $|\bar\nabla f|$ are really necessary in view of a result due to Pigola, Rigoli and Setti \cite{PigolaRigoliSetti}.

In \cite{CHHL} Casteras, Heinonen, Holopainen and Lira studied the asymptotic Dirichlet problem for Killing graphs with prescribed mean curvature on warped product manifolds $M\times_\varrho\R$. Here $M$ is a complete $n$-dimensional Riemannian manifold and $\varrho\in C^\infty(M)$ is a smooth warping function. This means that the metric of $M\times_\varrho\R$ can be written in the form
	\[
	(\varrho\circ\pi_1)^2 \pi^*_2 dt^2 + \pi^*_1g,
	\]
where $g$ is the metric on $M$ and $\pi_1\colon M\times\R \to M$ and $\pi_2 \colon M\times\R \to \R$ are the standard projections. In this setting $X=\p_t$ is a Killing field with $|X|=\varrho$ on $M$. The norm of $X$ is preserved along its flow lines and therefore $\varrho$ can be extended to a smooth function $\varrho=|X| \in C^\infty(M\times_\varrho\R)$. Then the Killing graph of a function $u\colon M\to \R$ is the hypersurface given by
	\[
	\Sigma_u = \{\Psi(x,u(x)) \colon x\in M\},
	\]
where $\Psi\colon M\times\R\to M\times_\varrho\R$ is the flow generated by $X$. Such Killing graphs were first introduced in \cite{DajRip} and \cite{DHL} (see also \cite{DLR}). In \cite{DHL} it was shown that the Killing graph $\Sigma_u$ has mean curvature $H$ if $u$ satisfies the equation
	\begin{equation}\label{kil-eq}
	\dv_{-\log\varrho} \frac{\nabla u}{\sqrt{\varrho^{-2} + |\nabla u|^2}} = nH,
	\end{equation}
where $\dv_{-\log\varrho} (\cdot) = \dv (\cdot) + \ang{\nabla \log\varrho, (\cdot)}$ is weighted divergence operator.

In the case of Killing graphs the solvability of the asymptotic Dirichlet problem depends on the geometry of $M$, on the warping function $\varrho$ and on the prescribed mean curvature $H$. In \cite{CHHL} the authors consider the same curvature assumptions \eqref{HoVaCurv2} and \eqref{HoVaCurv1} on $M$ as in \cite{CHR} and \cite{CHH2}. We point out that here it is necessary to assume that the warping function $\varrho$ is convex since otherwise the whole warped product space would not be a Cartan-Hadamard manifold, see \cite{BO}.

Depending on the curvature bounds on $M$ and the warping function $\varrho$, it is possible to find entire Killing graphs with prescribed mean curvature $H$ such that $H(x) \not\to 0$ as $r(x) \to \infty$. For example, the hyperbolic space $\HH^{n+1}$ with constant curvature $-1$, can be written as a warped product $\HH^{n+1} = \HH^n \times_{\cosh} \R$ and in this case the natural bound for the mean curvature function is
	\[
	|H|<1.
	\]
This comes from the fact that, in order to solve the Dirichlet problem, the prescribed mean curvature must be bounded from above by the mean curvature of the Killing cylinder over the domain.

To finish this section, we mention also the article \cite{BCKRT} that appeared after writing the first version of this survey. In \cite{BCKRT} Bonorino, Casteras, Klaser, Ripoll and Telichevesky study the asymptotic Dirichlet problem in $\HH^n \times\R$ with prescribed mean curvature $H\colon \HH^n\times\R \to \R$ depending also on the height variable.



\section{Different strategies to solve the asymptotic Dirichlet problem}

The proof of the solvability of the asymptotic Dirichlet problem is done in two steps. Namely, first one needs to obtain an entire solution, and secondly, show that this solution has the correct behaviour at infinity. The first step is easy (at least for the minimal graphs) in a sense that it follows from the known theory of PDE's: The Dirichlet problem for the minimal surface equation is solvable on bounded domains if the mean curvature of the boundary is nonnegative. To obtain an entire solution this can be used as follows.

Let $\theta\in C(\pinf M)$ be a continuous function defined on the asymptotic boundary. Fix a point $o\in M$ and extend $\theta$ to a function $\theta\in C(\bar M)$. Then for each $k\in\N$ there exists a solution $u_k \in C^{2,\alpha} (B(0,k)) \cap C(\bar B(o,k))$ to
	\[
	\begin{cases}
	\dv \frac{\nabla u_k}{\sqrt{1+|\nabla u_k|}} = 0 \quad \text{in } B(o,k) \\
	u_k|\p B(o,k) = \theta.
	\end{cases}
	\]
Applying interior gradient estimate in compact subsets of $M$ and using the diagonal argument, one finds a subsequence that converges locally uniformly with respect to $C^2$-norm to a solution $u$. Then the hard part is to show that $u$ extends continuously to $\pinf M$ and satisfies $u|\pinf M = \theta$, that is 
	\begin{equation}\label{eq-inf-val}
	\lim_{x\to x_0} u(x) = \theta(x_0)
	\end{equation}
	for any $x_0\in\pinf M$.
	In the case of operators mentioned in Section \ref{sec-other-op}, it is also necessary to obtain global height estimates to get the converging subsequence, since constants are no more solutions.

In order to complete the proof of the correct behaviour at infinity, a couple of different strategies have been applied. One possibility is to construct local barriers at infinity and use them to show that the solution must have the desired behaviour. Another possibility is to use more PDE related techniques, such as Caccioppoli and Poincar\'e-type inequalities, in order to obtain integral estimates and to show that the difference of the solution and boundary data is pointwise bounded by these integrals that tend to $0$ at infinity. It is this latter part where the geometry of the manifold, i.e. the curvature assumptions, play a key role.

Finally, the uniqueness of the solutions (for operators satisfying maximum principle) can be proved as follows. Assume that $u$ and $\tilde u$ are solutions, continuous up to the boundary, and $u=\tilde u$ on $\pinf M$. Assume that there exists $y\in M$ with $u(y) > \tilde u(y)$. Now denote $\delta = (u(y) - \tilde u(y))/2$ and let $U\subset \{x\in M \colon u(x) > \tilde u(x) +\delta\}$ be the component of $M$ containing the point $y$. Since $u$ and $\tilde u$ are continuous functions that coincides on the asymptotic boundary $\pinf M$, it follows that $U$ is relatively compact open subset of $M$. Moreover, $u = \tilde u +\delta$ on $\p U$, which implies $u = \tilde u+\delta$ in $U$. This is a contradiction since $y \in U$.


\subsection{Methods related to barriers}

First way to show \eqref{eq-inf-val} is to use local barriers at the given point $x_0$, i.e. to bound the values of the solution $u$ from above and from below by sub- and supersolutions that are approaching the desired value $\theta(x_0)$ when $x\to x_0$. The idea is as follows. Let $x_0\in \pinf M$ and $\ve>0$ be arbitrary. Then by the continuity of $\theta$ (the extended function in $\bar M$), there exists a neighbourhood $W$ of $x_0$ such that
	\[
	|\theta(x) - \theta(x_0)| < \ve/2
	\]
for all $x\in W$.

If we have a supersolution $\psi$ (in $W$) so that $\psi(x) \to 0$ as $x\to x_0$, the function $-\psi$ will be a subsolution and we aim to bound the solution $u$ in $W$ as
	\[
	-\psi(x) + \theta(x_0) -\ve \le u(x) \le \psi(x) +\theta(x_0) + \ve.
	\]
This follows since, by comparison principle, we can bound every solution $u_k$ (for $k$ large enough) to be between these barriers. Finally it follows that
	\[
	\limsup_{x\to x_0} |u(x) - \theta(x_0)| \le \ve
	\]
since $\lim_{x\to x_0} \psi(x) = 0$. Because $\ve$ and $x_0 \in\pinf M$ were arbitrary, this shows the claim.

\subsubsection{Strict convexity condition.}

Motivated by Choi's \cite{choi} convex conic neighbourhood condition, that was used for the Laplacian, Ripoll and Telichevesky \cite{RT} introduced the following strict convexity condition to suit more general divergence form quasilinear elliptic PDE's.
	\begin{defin}
		A Cartan-Hadamard manifold $M$ satisfies the strict convexity condition (SC condition) if, given $x\in\pinf M$ and a relatively open subset $W\subset \pinf M$ containing $x$, there exists a $C^2$ open subset $\Omega\subset M$ such that $x\in \Int \pinf \Omega \subset W$, where $\Int \pinf\Omega$ denotes the interior of $\pinf\Omega$ in $\pinf M$, and $M\setminus \Omega$ is convex.
\end{defin}
They showed that if $M$ satisfies the SC condition and the sectional curvatures of $M$ are bounded from above by a negative constant, $K_M\le -k^2$, it is possible to use the distance function, $s\colon \Omega\to \R$, to $\p\Omega$ to construct barriers at $x$. They also proved that rotationally symmetric manifolds with $K_M\le -k^2$ and manifolds satisfying \eqref{HoVaCurv2} satisfy the SC condition.

In order to prove the SC condition under \eqref{HoVaCurv2}, Ripoll and Teliche\-vesky generalise former constructions of Anderson \cite{andJDG} and Borbély \cite{borbpams}. The idea is that since $K_M\le -k^2$, the principal curvatures of the geodesic spheres are at least $k$, and therefore it is possible to take out small pieces of the spheres such that the remaining set is still convex.
 
Later, Casteras, Holopainen and Ripoll \cite{CHR2} proved the SC condition under the curvature bounds  \eqref{CHRoptub} and \eqref{CHR-lowb}. In order to do this, they used slightly modified version of the local barrier function that will be introduced next.


\subsubsection{Barrier from an angular function.}

In \cite{HoVa} Holopainen and Vähä\-kangas solved the asymptotic Dirichlet problem for the $p$-Laplacian under very general curvature conditions on the manifold $M$. A key tool was a local barrier function at infinity that was constructed by generalising the ideas that go back to Anderson and Schoen \cite{andschoen} and to Holopainen \cite{Ho}. The idea was to take continuous function on the boundary $\pinf M$, extend it to the whole manifold and after a smoothening procedure obtain sub- and supersolutions that can be used as barriers. The clever idea in \cite{HoVa} was that the smoothening procedure depends also on the curvature lower bound, and this allowed the more general curvature conditions \eqref{HoVaCurv2} and  \eqref{HoVaCurv1}.

The barrier function obtained in \cite{HoVa} has appeared to be very flexible and suit also other quasilinear elliptic PDE's. The same barrier has been used to solve the asymptotic Dirichlet problem for a large class of operators in \cite{CHR}, for $f$-minimal graphs in \cite{CHH2}, and for Killing graphs in \cite{CHHL}. As mentioned in the previous section, modified version of this barrier was used also in \cite{CHR2}.

We point out that since the smoothening procedure in \cite{HoVa} depends on the curvature bounds, the computations become very long and technical. However, to give a very brief idea, we describe some steps. The barrier function is constructed under curvature bounds
	\[
	-(b\circ r)^2(x) \le K(P_x) \le -(a\circ r)^2(x),
	\]
where $r$ is the distance to a fixed point $o\in M$ and $a,b\colon [0,\infty) \to [0,\infty)$ are smooth functions satisfying certain conditions (see \cite{HoVa}). These curvature bounds are used to control the first two derivatives of the ``barrier" that is constructed for each boundary point $x_0\in\pinf M$. 

The construction starts from an angular function $h$ that is defined on the boundary $\pinf M$, then extended inside the manifold, smoothened by integrating against certain kernel, and appropriately normalised. We denote this smooth function still by $h$. Then adding this function $h$ to the distance function $r$ with certain negative power and multiplying by a constant we will obtain a function $\psi$ that is a supersolution in a sufficiently small neighbourhood of $x_0$ in the cone topology. Finally this supersolution can be used to construct local barriers in that neighbourhood.


\subsubsection{Rotationally symmetric manifolds.}\label{intro_rotsym}

 On rotationally symmetric manifolds, with the metric
 	\[
	g^2 = dr^2 + f(r)^2 d\vartheta^2,
	\]
the solvability of the asymptotic Dirichlet problem under the optimal curvature bounds of Theorem \ref{thm1.5} is obtained by proving first the following integral condition that can be used to construct global barrier functions. 
 \begin{thm} \label{rotsymmingraph}
Assume that 
	\begin{equation}\label{intcond}
	\int_1^{\infty}\Big( f(s)^{n-3} \int_s^{\infty} f(t)^{1-n} dt \Big) ds < \infty.
	\end{equation}
	Then there exist non-constant bounded solutions of the minimal surface equation and, moreover,
	the asymptotic Dirichlet problem for the minimal surface equation is uniquely solvable for any 
	continuous boundary data on $\pinf M$.
\end{thm} 

The condition \eqref{intcond} was earlier considered by March in \cite{march} where, by studying the behaviour of the Brownian motion, he proved that entire bounded nonconstant harmonic functions exist if and only if \eqref{intcond} is satisfied. Similar conditions appeared also in \cite{milnor} and \cite{choi}. We will sketch the proof for the minimal surface equation from \cite{CHH3}. How the condition \eqref{intcond} gives the curvature bounds of Theorem \ref{thm1.5} can be found from \cite{march}.

\begin{proof}[Proof of Theorem \ref{rotsymmingraph} (Sketch)]

Changing the order of integration, the condition \eqref{intcond} reads
        \begin{equation}\label{intcond2}
                \int_1^\infty \frac{\int_1^t f(s)^{n-3} ds}{f(t)^{n-1}} dt <\infty.
        \end{equation}
Now interpret $\pinf M$ as $\Ss^{n-1}$ and let $b\colon \Ss^{n-1} \to \R$ be a smooth non-constant
function and define $B\colon M \setminus\{o\} \to \R$,
\[
 B(\exp(r\vartheta))=B(r,\vartheta) = b(\vartheta), \, \vartheta \in \Ss^{n-1}\subset T_oM.
 \] 
 Define also
        \[
                \eta(r) = k \int_r^\infty f(t)^{-n+1} \int_1^t f(s)^{n-3} ds dt,
        \]
with $k>0$ to be determined later, and note that by the assumption \eqref{intcond2} $\eta(r) \to 0$ as $r\to\infty$. The idea in the proof is to use the functions $\eta$ and $B$, and condition \eqref{intcond2} to construct global barrier functions.

The minimal surface equation for $\eta + B$ can be written as
    \begin{align}\label{mineqetaB}\begin{split}
      \dv \frac{\nabla(\eta+B)}{\sqrt{1+|\nabla(\eta+B)|^2}} &=
        \frac{\Delta (\eta+B)}{\sqrt{1+|\nabla(\eta+B)|^2}} \\
     &\qquad + \ang{\nabla(\eta+B), \nabla\Big(\frac{1}{\sqrt{1+|\nabla(\eta+B)|^2}}\Big)},\end{split}
    \end{align}
and we want to estimate the terms on the right hand side from above. 

An important fact is that on the rotationally symmetric manifolds the Laplace operator can be written as
    \begin{equation}\label{laplacepolar}
     \Delta = \frac{\p^2}{\p r^2} + (n-1) \frac{f'\circ r}{f\circ r} \frac{\p}{\p r} + 
     \frac{1}{(f\circ r)^2} \Delta^{\Ss},
    \end{equation}
where $\Delta^{\Ss}$ is the Laplacian on the unit sphere $\Ss^{n-1}\subset T_oM$, and for the gradient of a function $\varphi$ 
we have
    \begin{equation}\label{gradientpolar}
     \nabla \varphi = \frac{\p \varphi}{\p r} \frac{\p}{\p r} + \frac{1}{f(r)^2}\nabla^{\Ss}\varphi 
    \end{equation}
and
    \[
     |\nabla \varphi|^2 = \varphi_r^2 + f^{-2} |\nabla^{\Ss}\varphi|^2.
    \]
Here $\nabla^{\Ss}$ is the gradient on $\Ss^{n-1}$, $|\nabla^{\Ss}\varphi|$ denotes the norm of $\nabla^{\Ss}\varphi$ with respect to the Euclidean metric on $\Ss^{n-1}$,
 and $\varphi_r=\partial\varphi/\partial r$.

Therefore using \eqref{mineqetaB}, \eqref{laplacepolar} and \eqref{gradientpolar} we obtain, by somewhat long computation, 
        \begin{equation}\label{etaBestim}
        		\dv \frac{\nabla(\eta+B)}{\sqrt{1+|\nabla(\eta+B)|^2}} \le 0
	\end{equation} 
when we choose $r$ large enough and then $k \ge ||b||_{C^2}$ large enough. In particular, $\eta + B$ is a supersolution to the minimal surface equation in $M\setminus B(o,r_0)$ for some $r_0$.

To obtain a global upper barrier choose $k$ so that \eqref{etaBestim} holds and $\eta  > 2 \max |B|$ on the geodesic sphere $\p B(o,r_0)$. Then $a \coloneqq \min_{\p B(o,r_0)} (\eta +  B) >  \max B$. Since $\eta(r)\to 0$ as $r\to\infty$, the function 
        \[
                w(x) \coloneqq \begin{cases}
                        \min\{(\eta +  B)(x) , a \} &\text{ if } x\in M \setminus B(o,r_0); \\
                        a &\text{ if } x\in B(o,r_0)
                \end{cases}
        \]
is continuous in $\bar{M}$ and coincide with $b$ on $\pinf M$. Global lower barrier $v$ can be obtained similarly by replacing $\eta$ with $-\eta$. Then $v\le B\le w$ by construction and the same will hold also for the solution $u$. In the end, continuous boundary values can be handled by approximation.

\end{proof}


\subsection{Sobolev and Poincar\'e inequalities with Moser iteration}

The method of proving the correct boundary values at infinity without barriers goes back to \cite{cheng} where Cheng considered the asymptotic Dirichlet problem for the harmonic functions. Cheng's approach was modified by Vähäkangas in \cite{Va1} and \cite{Va2} for the case of $\mathcal{A}$-harmonic functions. It turned out that Cheng's and Vähäkangas' proofs could be further developed to suit also the minimal surface equation and weaker assumptions on the curvature (\cite{CHH1}, \cite{CHR1}). We will describe some steps of the proof of Theorem \ref{pinch_maintheorem}; the proof of Theorem \ref{CHR-op-thm} follows similar steps although details differ due to the different curvature assumptions.

Within this approach we will deal with weak solutions of the minimal surface equation \eqref{mingraph-eq} that are defined as follows.
Let $\Omega\subset M$ be an open set. Then a function $u\in W^{1,1}_{\loc}(\Omega)$ is a
 \emph{(weak) solution of the minimal surface equation} if
    \begin{equation}\label{MinEqWeak}
      \int_{\Omega} \frac{\ang{\nabla u,\nabla \varphi}}{\sqrt{1+\abs{\nabla u}^2}} = 0
    \end{equation}
 for every $\varphi\in C_0^{\infty}(\Omega)$. Note that the integral is well-defined since
    \[
      \sqrt{1+\abs{\nabla u}^2} \ge \abs{\nabla u} \quad \text{a.e.},
    \]
 and thus
    \[
      \int_{\Omega} \frac{\abs{\ang{\nabla u,\nabla \varphi}}}{\sqrt{1+\abs{\nabla u}^2}} \le
       \int_{\Omega} \frac{\abs{\nabla u}\abs{\nabla \varphi}}{\sqrt{1+\abs{\nabla u}^2}} \le
       \int_{\Omega} \abs{\nabla \varphi} < \infty.
    \]

Now let $\theta$ be the boundary data function that is extended inside the manifold $M$ and let $u$ be the solution obtained via the method mentioned in the beginning of this section. In order to show \eqref{eq-inf-val}, we construct a certain smooth homeomorphism $\varphi\colon[0,\infty) \to [0,\infty)$ (see \eqref{psi-phi-constr}) and prove first a pointwise estimate
	\begin{equation}\label{eq-pt-est}
      \sup_{B(x,s/2)} \varphi \big(\abs{u-\theta}/\nu\big)^{n+1} \le c\int_{B(x,s)} \varphi 
      \big(\abs{u-\theta}/\nu\big),
	\end{equation}
where $c$ and $\nu$ are constants. Second step is to show that the integral on the right hand side of \eqref{eq-pt-est} tends to $0$ as $r(x)\to\infty$, which then implies that the pointwise difference of the boundary data and the solution tends to $0$. This second step follows from an integral estimate (Poincar\'e-type inequality)
	\begin{equation}\label{eq-int-est}
	\int_{M} \varphi(|u-\theta|/\nu) \le c + c\int_{M} F(r |\nabla\theta|) + 
	  c\int_{M} F_1( r^2 |\nabla\theta|^2) < \infty.
	\end{equation}
Namely, let now $(x_i)$ be a sequence of points such that $x_i\to x_0 \in \pinf M$. Since the integral of $\varphi$ over the whole manifold is finite by \eqref{eq-int-est}, we must have
	\[
	\int_{B(x_i,s)} \varphi (|u-\theta|/\nu) \to 0
	\]
as $x_i \to x_0$ and \eqref{eq-inf-val} follows.

Key tools to obtain the pointwise estimate \eqref{eq-pt-est} are the Sobolev inequality (see e.g. \cite{HofSp})
	\begin{equation}\label{sobolev_eq}
	\left(\int_{B(x,r_s)} \abs{\eta}^{n/(n-1)}\right)^{(n-1)/n} \le C_s \int_{B(x,r_s)} \abs{\nabla \eta},
	\end{equation}
	$\eta\in C_0^\infty(B(x,r_s))$,
and a Caccioppoli-type inequality \cite[Lemma 3.1]{CHH1} 
	\begin{multline} \label{caceq1}
  \int_B \eta^2 \varphi'(|u-\theta|/\nu) \frac{\abs{\nabla u}^2}{\sqrt{1+\abs{\nabla u}^2}} 
  \le C_{\ve} \int_B \eta^2 \varphi'(|u-\theta|/\nu) \abs{\nabla \theta}^2  \\+  (4+\ve)\nu^2 \int_B
  \frac{\varphi^2}{\varphi'}(|u-\theta|/\nu) \abs{\nabla \eta}^2,
  \end{multline}
  where $\eta\ge0$ is a $C^1$ function. We note that the inequality  \eqref{caceq1} can be proved by using
    \[
      f = \eta^2 \varphi\left( \frac{(u-\theta)^+}{\nu}\right) - \eta^2 \varphi \left(
      \frac{(u-\theta)^-}{\nu} \right)
    \]
as a test function in the weak form of the minimal surface equation \eqref{MinEqWeak}. These inequalities can be then used to run a Moser iteration process that yields the estimate \eqref{eq-pt-est}. The idea of the Moser iteration, that goes back to \cite{Moser61} and \cite{Serrin64} (see also \cite{GT}), is that if one has an estimate for $L^p$-norms, i.e. $L^{p'}$-norm can be suitably bounded by $L^{p}$-norm when $p'\ge p$, then it is possible to iterate this estimate and obtain a bound for the $L^\infty$-norm in terms of the $L^p$-norm with finite $p$. 
In our case we first obtain an estimate that can be written as a recursion formula $I_{j+1} \le C^{1/\kappa^j} \kappa^{j/\kappa^j} I_j$, where
	\[
	I_j = \left( \int_{B_j} \varphi(h)^{m_j} \right)^{1/\kappa^j},
	\]
$\kappa= n/(n-1)$, $j\in\N$, $m_j=(n+1)\kappa^j -n$ and the radii of the balls $B_j$ converge to $s/2$ as $j\to \infty$, and this finally yields \eqref{eq-pt-est}.

The Caccioppoli inequality \eqref{caceq1}, together with the curvature bound, play a central role also in the proof of the integral estimate \eqref{eq-int-est}. Another essential part is the Young's inequality
	\[
	ab \le F(a) + G(b)
	\]
for special type of Young functions $F$ and $G$ that are constructed as follows. Let $H\colon [0,\infty) \to [0,\infty)$ be a certain homeomorphism (for details, see \cite[Section 2.3]{CHH1}) and define $G(t) = \int_0^t H(s) ds$ and $F(t) = \int_0^t H^{-1}(s) ds$. Then
	\begin{equation}\label{psi-phi-constr}
	\psi(t) = \int_0^t \frac{ds}{G^{-1}(s)} \quad \text{and} \quad \varphi = \psi^{-1}
	\end{equation}
are homeomorphisms so that $G\circ\varphi' = \varphi$. Another pair of Young functions $F_1$ and $G_1$ are constructed similarly, and so that $G_1\circ \varphi'' \approx \varphi$. The integrability of the functions $F$ and $F_1$ in \eqref{eq-int-est} follows from the construction and the curvature assumptions on $M$. 


\section{Non-solvability of the asymptotic Dirichlet problem}\label{intro_nonex}

From the classical results on Bernstein's problem we know that if the graph of $u\colon\R^n\to \R$ is a minimal surface in $\R^{n+1}$, then $u$ is affine for $n\le 7$, and so a bounded solution must be constant. Therefore it is clear, that if we wish to solve the asymptotic Dirichlet problem with any continuous boundary data, the curvature cannot be zero everywhere.
On the other hand, the discussion about the rotationally symmetric case and the theorems replacing the sectional curvature lower bound with the pinching condition \eqref{pinchassump} raise a question about the necessity of the curvature lower bound. It turns out that for the solvability of the asymptotic Dirichlet problem for \eqref{mingraph-eq} on general non-rotationally symmetric Cartan-Hadamard manifold, some control on the negativity of the curvature is necessary (e.g. the condition \eqref{pinchassump} or a curvature lower bound).

The first result in this direction was proved in 1994 by Ancona in \cite{ancrevista} where he constructed a manifold with $K_M\le -1$ so that the Brownian motion almost surely exits $M$ at a single point  on the asymptotic boundary, and therefore the asymptotic Dirichlet problem for the Laplacian is not solvable.
Borbély \cite{Bor} constructed similar manifold using analytic arguments and later Holopainen and Ripoll \cite{HR_ns} generalised Borbély's example to cover also the minimal surface equation by proving the following. 
\begin{thm}
There exists a $3$-dimensional Cartan-Hadamard manifold $M$ with sectional curvatures $\le-1$ such that the asymptotic Dirichlet problem for the minimal surface equation \eqref{mingraph-eq} is not solvable with any continuous nonconstant boundary data, but there are nonconstant bounded continuous solutions of \eqref{mingraph-eq} on $M$. 
\end{thm}

It is also worth mentioning two closely related results by Greene and Wu \cite{GWgap} that partly answer the question about the optimal curvature upper bound. Firstly, in \cite[Theorem 2 and Theorem 4]{GWgap} they showed that an $n$-dimensional, $n\neq2$, Cartan-Hadamard manifold with asymptotically nonnegative sectional curvature is isometric to $\R^n$. Secondly, in \cite[Theorem 2]{GWgap} they showed that an odd dimensional Riemannian manifold with a pole $o\in M$ and everywhere non-positive or everywhere nonnegative sectional curvature is isometric to $\R^n$ if $\liminf_{s\to\infty} s^2 k(s) = 0$, where $k(s) = \sup\{ |K(P_x)| \colon x\in M, \, d(o,x)=s, P_x\in T_xM \, \text{two-plane}\}$.
 
Above the asymptotically nonnegative sectional curvature means the following.
\begin{defin}\label{ansc-def}
Manifold $M$ has asymptotically nonnegative sectional curvature if there exists a continuous decreasing function $\lambda\colon [0,\infty) \to [0,\infty)$ such that
	\[
	\int_0^\infty s\lambda(s) \, ds < \infty,
	\]
and that $K_M(P_x) \ge - \lambda(d(o,x))$ at any point $x\in M$.
\end{defin}
In \cite{CHH3} Casteras, Heinonen and Holopainen proved the following under the conditions of Definition \ref{ansc-def}.
\begin{thm}
Let $M$ be a complete Riemannian manifold with asymptotically nonnegative sectional curvature and only one end. If $u\colon M\to\R$ is a solution to the minimal surface equation that is bounded from below and has at most linear growth, then it must be a constant.
\end{thm}
Compared to the results of Greene and Wu, here $M$ does not need to be simply connected and the curvature is allowed to change sign. Similar result was proved also in \cite{RSS} but assuming nonnegative Ricci curvature. To end the discussion, we point out that for example the curvature lower bound
	\begin{equation*}\label{ansc-bound}
	 K(P_x) \ge -\frac{C}{r(x)^2 \big(\log r(x)\big)^{1+\ve}}, \quad C>0,
	\end{equation*}
satisfies Definition \ref{ansc-def}, and this should be compared to the bounds \eqref{CHRoptub} and those of Theorem \ref{thm1.5}.



\end{document}